\documentclass[11pt,reqno]{amsart}
\usepackage{graphicx}
\usepackage[pdftex]{hyperref}
\usepackage{subfigure}
\usepackage{amsthm}
\usepackage[margin=2cm]{geometry}
\usepackage{palatino}

\newtheorem{theorem}{Theorem}[section]
\newtheorem{lemma}[theorem]{Lemma}

\theoremstyle{definition}

\theoremstyle{remark}
\newtheorem{remark}[theorem]{Remark}

\numberwithin{equation}{section}

\begin{document}

\title{A monotonicity property of a new Bernstein type operator}

\author{Floren\c{t}a Trip\c{s}a}
\email{florentatripsa@yahoo.com}
\address{Faculty of Mathematics and Computer Science, Transilvania University of Bra\c{s}ov, Str. Iuliu Maniu 50, Bra\c{s}ov -- 500091, Romania.}

\author{Nicolae R. Pascu}
\address{Department of Mathematics, Kennesaw State University, 1100 S. Marietta Parkway, Marietta, GA 30060-2896, U.S.A.}
\email{npascu@kennesaw.edu}

\begin{abstract} 
In the present paper we prove that the probabilities of the P\'{o}lya urn distribution (with negative replacement) satisfy a monotonicity property similar to that of the binomial distribution (P\'{o}lya urn distribution with no replacement).

As a consequence, we show that the random variables with P\'{o}lya urn distribution (with negative replacement) are stochastically ordered with respect to the parameter giving the initial distribution of the urn. An equivalent formulation of this result shows that the new Bernstein operator recently introduced in \cite{PaPaTr} is a monotone operator.

The proofs are probabilistic in spirit and rely on various inequalities, some of which are of independent interest (e.g. a refined version of the reversed Cauchy-Bunyakovsky-Schwarz inequality).
\end{abstract}

\subjclass[2008]{Primary 41A36, 41A25, 41A20.}

\keywords{P\'{o}lya urn distribution, Bernstein operator, monotone operator.}

\maketitle

\section{Introduction}
P\'{o}lya urn model (also known as Pólya-Eggenberger urn model) is an experiment in which one observes the number of white balls extracted from an urn containing initially $a$ white balls and $b$ black balls, when the extracted ball is replaced in the urn together with $c$ balls of the same color before the next extraction from the urn.

Denoting by $X_n^{a,b,c}$ the random variable representing the number of white balls obtained in $n\ge1$ extractions from the urn, it can be shown (e.g. \cite{Kotz}) that the model is well defined (defines a distribution) for $a,b>0$ and $c\in\mathbb{R}$ satisfying $c\ge -\min\{a,b\}/{(n-1)}$, and the distribution is given by
\begin{equation}\label{probabilities p_n,k}
p_{n,k}^{a,b,c}=P\left( X_{n}^{a,b,c}=k\right) =C_{n}^{k}\frac{a^{\left(
k,c\right) }b^{\left( n-k,c\right) }}{1^{\left( n,c\right) }}, \qquad k\in\{0,1,\ldots,n\},
\end{equation}%
where $x^{\left( 0,h\right) }=1$, $x^{\left( k,h\right) }=x\left( x+h\right)
\cdot \ldots \cdot \left( x+\left( k-1\right) h\right) $ for $k\geq 1$,
denotes the rising factorial with increment $h$.

In the present paper we focus on the case $a=x\in\left[0,1\right]$, $b=1-x$, and the minimal choice of the replacement parameter, $c=-\min\{x,1-x\}/(n-1)$. The reason for this choice is two-fold. The first is that it is an interesting problem of study from the probabilistic point of view, and the second relates to the newly introduced operator $R_n$ (\cite{PaPaTr, PaPaTr2}) defined by
\begin{equation}\label{operator R_n}
R_{n}\left( f,x\right) =Ef\left( \frac{1}{n}X_{n}^{x,1-x,-\min \left\{
x,1-x\right\} /\left( n-1\right) }\right).
\end{equation}%

We show that in this case the corresponding probabilities satisfy a monotonicity property with respect to the initial urn distribution, similar to that of the binomial distribution (the case of the replacement parameter $c=0$), and as a consequence, we show that the random variables with P\'{o}lya urn distribution satisfy a natural stochastic ordering. In an equivalent formulation, this results shows that the opera\-tor $R_n$ has a shape preserving property (monotonicity property), similar to that of the classical Bernstein operator (the case $c=0$, see e.g. \cite{Bustamante}).

The structure of the paper is the following. Section \ref{Auxiliary results} contains some auxiliary results of independent interest, needed in the sequel. In Lemma \ref{Reversed Cauchy inequality} we prove an interesting inequality, which may be seen as a refined version of a reversed Cauchy-Bunyakovsky-Schwarz inequality (see Remark \ref{Reversed Cauchy inequality}). In Lemma \ref{lemma for sum bounds} we give estimates for sums involving functions with three positive continuous derivatives. Lemma \ref{Localization lemma for x_n,k} is a technical result concerning the sign of a certain function, essential for proving our main results.

In Section \ref{Main results} we first prove that the P\'{o}lya urn probabilities satisfy a certain monotonicity property with respect to the initial urn distribution, similar to that of the binomial distribution (Theorem \ref{theorem on monotonicity of probabilities}). Using this, in Theorem \ref{Stochastic ordering of Polya rv} we show that the P\'{o}lya random variables satisfy a natural stochastic ordering. In an equivalent formulation (Theorem \ref{R_n monotone operator}), this results shows that the operator $R_n$ is, similarly to the classical Bernstein operator, a monotone operator.

\section{Auxiliary results}\label{Auxiliary results}

We begin with the following auxiliary result of independent interest.

\begin{lemma}\label{Reversed Cauchy inequality}
For integers $n\geq 2$ and $k\in
\left\{ 1,\ldots ,n-1\right\} $, and positive real numbers $a_{1},\ldots
,a_{n},b_{1},\ldots ,b_{n-k}$ with $\max_{1\leq i\leq n}a_{i}\leq
\min_{1\leq j\leq n-k}b_{j}$ and $\sum_{i=1}^{n}a_{i}=\sum_{j=1}^{n-k}b_{j}$%
, we have $$\sum_{i=1}^{n}a_{i}^{2}<\sum_{j=1}^{n-k}b_{j}^{2}.$$
\end{lemma}

\begin{proof}
Without loss of generality we may assume assume $a_{1}\leq \ldots \leq
a_{n}\leq b_{1}\leq \ldots \leq b_{n-k}$.

Note that $\displaystyle\frac{\sum_{i=1}^{n}a_{i}}{n-k}=\frac{\sum_{k=1}^{n-k}b_{j}}{n-k}\geq
b_{1}\geq a_{n}$,  and therefore $\displaystyle a_{i}\leq a_{n}\leq \frac{\sum_{i=1}^{n}a_{i}}{n-k}$, for $i\in \left\{
1,\ldots ,n\right\}$.

We obtain%
\begin{equation}
\sum_{i=1}^{n}a_{i}^{2}=\sum_{i=1}^{n}\left( a_{i}\cdot a_{i}\right) \leq
\frac{1}{n-k}\sum_{i=1}^{n}\left( a_{i}\sum_{j=1}^{n}a_{j}\right) =\frac{1}{%
n-k}\left( \sum_{i=1}^{n}a_{i}\right) ^{2},  \label{aux 7}
\end{equation}%
and using Cauchy-Bunyakovsky-Schwarz inequality we conclude%
\[
\sum_{i=1}^{n}a_{i}^{2}\leq \frac{1}{n-k}\left( \sum_{i=1}^{n}a_{i}\right)
^{2}=\frac{1}{n-k}\left( \sum_{j=1}^{n-k}b_{j}\right) ^{2}\leq
\sum_{j=1}^{n-k}b_{j}^{2}.
\]

We have left to show that the inequality is a strict inequality. If the
inequality in the statement of the lemma were an equality, from the proof
above we conclude that we must have $a_{1}=\ldots =a_{n}=a$ and $%
b_{1}=\ldots =b_{n}=b$. The hypothesis of the lemma becomes in this case $%
na=\left( n-k\right) b$, and in turn, this shows $%
\sum_{i=1}^{n}a_{i}^{2}=na^{2}=\frac{1}{n}\left( n-k\right) ^{2}b^{2}=\frac{%
n-k}{n}\sum_{j=1}^{n-k}b_{j}^{2}<$ $\sum_{j=1}^{n-k}b_{j}^{2}$, thus the equality cannot hold,
concluding the proof.
\end{proof}

\begin{remark}{Remark on reversed CBS inequality}
The inequality (\ref{aux 7}) in the proof above is a particular case of a
reversed Cauchy-Bunyakovsky-Schwarz inequality. For example, the P\'{o}%
lya-Szeg\"{o}'s inequality for sequences of positive numbers is%
\begin{equation}
\frac{\sum_{i=1}^{n}a_{i}^{2}\sum_{i=1}^{n}b_{i}^{2}}{\left(
\sum_{i=1}^{n}a_{i}b_{i}\right) ^{2}}\leq \frac{1}{4}\left( \sqrt{\frac{%
M_{1}M_{2}}{m_{1}m_{2}}}+\sqrt{\frac{m_{1}m_{2}}{M_{1}M_{2}}}\right) ^{2},
\end{equation}%
where $0<m_{1}\leq a_{i}\leq M_{1}<\infty $ and $0<m_{2}\leq b_{i}\leq
M_{2}<\infty $, $i\in \left\{ 1,\ldots ,n\right\} $. Taking $b_{1}=\ldots
=b_{n}=1$ (thus $m_{2}=M_{2}=1$, $m_{1}=a_{1}$, $M_{2}=a_{n}$) we obtain%
\[
\sum_{i=1}^{n}a_{i}^{2}\leq \frac{1}{4}\left( \sqrt{\frac{a_{n}}{a_{1}}}+%
\sqrt{\frac{a_{1}}{a_{n}}}\right) ^{2}\left( \sum_{i=1}^{n}a_{i}\right)
^{2},
\]%
for any positive sequence $a_{1}\leq a_{2}\leq \ldots \leq a_{n}$. Since $%
\frac{1}{4}\left( \sqrt{\frac{a_{n}}{a_{1}}}+\sqrt{\frac{a_{1}}{a_{n}}}%
\right) ^{2}\geq 1$, we see that the inequality (\ref{aux 7}) improves the P%
\'{o}lya-Szeg\"{o} inequality (under the additional hypotheses in the
statement of the lemma).
\end{remark}

We will also need the following auxiliary result.

\begin{lemma}\label{lemma for sum bounds}
Suppose $f\in C^{3}\left( \left[ 0,1\right]
\right) $ is such that $f,f^{\prime },f^{\prime \prime },f^{\prime \prime
\prime }\geq 0$ on $\left[ 0,1\right] $. Then for any integer $N\geq 1$ we
have%
\begin{equation}
0\leq \sum_{i=0}^{N}f\left( \frac{i}{N}\right) -N\int_{0}^{1}f\left( t\right) dt-\frac{f\left( 0\right) +f\left( 1\right) }{2%
}\leq \frac{f^{\prime
}\left( 1\right) -f^{\prime }\left( 0\right) }{4N}  \label{aux lemma}
\end{equation}
\end{lemma}

\begin{proof}
  \begin{figure}[h!]
  \caption{Trapezoids $T_i$ and $T_i^\prime$ in the proof of Lemma \ref{lemma for sum bounds}: 
      $\text{Area} (T_i^\prime) \le \int_{x_i}^{x_{i+1}} f(x)dx\le \text{Area} (T_i)$.}
      \includegraphics{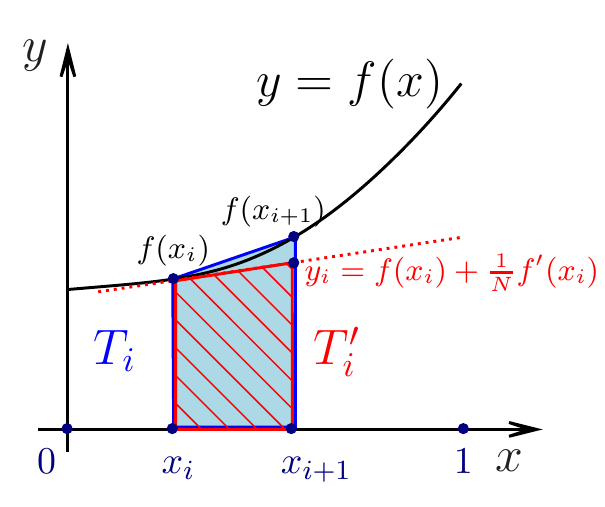}
      \label{fig1}
      \end{figure}
      
Since $f$ is convex, the sum of the areas of trapezoids $T_{0},\ldots ,T_{N-1}$
(see Figure \ref{fig1}) is larger than the area under the graph of $f$, thus%
\[
\int_{0}^{1}f\left( t\right) dt\leq \sum_{i=0}^{N-1}\frac{f\left(
x_{i}\right) +f\left( x_{i+1}\right) }{2}\cdot \frac{1}{N},
\]%
where $x_{i}=\frac{i}{N}$, $i\in \left\{ 0,1,\ldots ,N\right\} $.
Rearranging the terms of the sum we obtain%
\begin{equation}
N\int_{0}^{1}f\left( t\right) dt\leq -\frac{f\left( 0\right) +f\left(
1\right) }{2}+\sum_{i=0}^{N}f\left( \frac{i}{N}\right) ,  \label{aux5}
\end{equation}%
which proves the left inequality in (\ref{aux lemma}).



Since $f$ is convex, the tangent line to the graph of $f$ at $x_{i}$ lies
below the graph of $f$, and therefore the sum of the areas of corresponding
trapezoids $T_{i}^{^{\prime }}$ is smaller than the area under $f$ (see
Figure \ref{fig1}). Summing over $i\in \left\{ 0,1,\ldots ,N-1\right\} $ we
obtain%
\[
\sum_{i=0}^{N-1}\frac{f\left( x_{i}\right) +y_{i}}{2}\cdot \frac{1}{N}\leq
\int_{0}^{1}f\left( t\right) dt,
\]%
where $y_{i}=f\left( x_{i}\right) +f^{\prime }\left( x_{i}\right) \cdot
\frac{1}{N}$, or equivalent%
\[
N\int_{0}^{1}f\left( t\right) dt\geq \sum_{i=0}^{N-1}\frac{f\left(
x_{i}\right) +f\left( x_{i}\right) +\frac{f^{\prime }\left( x_{i}\right) }{N}%
}{2}=\sum_{i=0}^{N-1}f\left( x_{i}\right) +\frac{1}{2N}\sum_{i=0}^{N-1}f^{%
\prime }\left( x_{i}\right) ,
\]%
and therefore%
\[
\sum_{i=0}^{N}f\left( x_{i}\right) \leq f\left( 1\right) +\frac{1}{2N}%
f^{\prime }\left( 1\right) +N\int_{0}^{1}f\left( t\right) dt-\frac{1}{2N}%
\sum_{i=0}^{N}f^{\prime }\left( x_{i}\right) .
\]

Using the inequality (\ref{aux5}) with $f$ replaced by $f^{\prime }$, we
obtain%
\begin{eqnarray*}
&&\sum_{i=0}^{N}f\left( x_{i}\right) \leq f\left( 1\right) +\frac{1}{2N}%
f^{\prime }\left( 1\right) +N\int_{0}^{1}f\left( t\right) dt-\frac{1}{2N}%
\sum_{i=0}^{N}f^{\prime }\left( x_{i}\right) \\
&\leq &f\left( 1\right) +\frac{1}{2N}f^{\prime }\left( 1\right)
+N\int_{0}^{1}f\left( t\right) dt-\frac{1}{2N}\left( N\int_{0}^{1}f^{\prime
}\left( t\right) dt+\frac{f^{\prime }\left( 0\right) +f^{\prime }\left(
1\right) }{2}\right) \\
&=&N\int_{0}^{1}f\left( t\right) dt+\frac{f\left( 0\right) +f\left( 1\right)
}{2}+\frac{f^{\prime }\left( 1\right) -f^{\prime }\left( 0\right) }{4N},
\end{eqnarray*}%
concluding the proof.
\end{proof}

The following technical result is essential for the proof of our main results in the following section.

\begin{lemma}\label{Localization lemma for x_n,k}
For any integers $n\geq 2$ and $k\in
\left\{ 0,\ldots ,n-1\right\} $, there exists $x_{n,k}\in \left[ \frac{k-1}{%
n-1},\frac{k}{n-1}\right] $ such that the function%
\begin{equation}\label{definition of function varphi}
\varphi _{n,k}\left( x\right) =\sum_{i=0}^{n-1}\frac{1}{1-\frac{i}{n-1}x}%
-\sum_{i=0}^{n-k-1}\frac{1}{1-x-\frac{i}{n-1}x}, \quad x\in \left[0,\frac{n-1}{2n-k-2}\right),
\end{equation}%
is positive on $\left[0,x_{n,k}\right)$ and
negative on $\left(x_{n,k},\frac{n-1}{2n-k-2}\right)$.
\end{lemma}

\begin{proof}
Under the above hypotheses on $n$ and $k$, it is easy to verify that $\frac{k-1}{%
n-1}<\frac{k}{n}<\frac{n-1}{2n-k-2}\leq 1$ (the last inequality is a strict inequality if $k<n-1$).

If $k=0$, since $\frac{1}{1-x-\frac{i}{n-1}x}>\frac{1}{1-\frac{i}{n-1}x}$
for $i\in \left\{ 0,\ldots ,n-1\right\} $, we have $\varphi _{n,0}\left(
x\right) <0$ for $x\in (0,\frac{n-1}{2n-2})$, and the claim holds true with $%
x_{n,0}=0\in \left[ -\frac{1}{n-1},0\right] $.

If $k=n-1$, $\varphi _{n,k}\left( x\right) =\sum_{i=0}^{n-2}\frac{1}{1-\frac{%
i}{n-1}x}>0$ for $x\in \lbrack 0,\frac{n-1}{2n-\left( n-1\right) -2})=[0,1)$%
, and the claim holds true with $x_{n,n-1}=1\in \left[ \frac{n-2}{n-1},1%
\right] $.

Assume now $n>2$ and $k\in \left\{ 1,\ldots ,n-2\right\} $. We will first
show that if $\varphi _{n,k}\left( x\right) =0$, 
then $\varphi _{n,k}^{^{\prime }}\left( x\right) <0$
(note that $\varphi _{n,k}\left( 0\right) =k>0$, thus $x\neq 0$).

We have%
\begin{eqnarray*}
\varphi _{n,k}^{^{\prime }}\left( x\right) &=&\sum_{i=0}^{n-1}\frac{\frac{i}{%
n-1}}{\left( 1-\frac{i}{n-1}x\right) ^{2}}-\sum_{i=0}^{n-k-1}\frac{1+\frac{i%
}{n-1}}{\left( 1-x-\frac{i}{n-1}x\right) ^{2}} \\
&=&\frac{1}{x}\left( \sum_{i=0}^{n-1}\frac{-\left( 1-x\frac{i}{n-1}\right) +1%
}{\left( 1-\frac{i}{n-1}x\right) ^{2}}-\sum_{i=0}^{n-k-1}\frac{-\left(
1-x\left( 1+\frac{i}{n-1}\right) \right) +1}{\left( 1-x-\frac{i}{n-1}%
x\right) ^{2}}\right) \\
&=&-\frac{1}{x}\varphi _{n,k}\left( x\right) +\frac{1}{x}\left(
\sum_{i=0}^{n-1}\frac{1}{\left( 1-\frac{i}{n-1}x\right) ^{2}}%
-\sum_{i=0}^{n-k-1}\frac{1}{\left( 1-x-\frac{i}{n-1}x\right) ^{2}}\right) .
\end{eqnarray*}

If $\varphi _{n,k}\left( x\right) =0$, we obtain $\varphi _{n,k}^{^{\prime
}}\left( x\right) =\frac{1}{x}\left( \sum_{i=0}^{n-1}\frac{1}{\left( 1-\frac{%
i}{n-1}x\right) ^{2}}-\sum_{i=0}^{n-k-1}\frac{1}{\left( 1-x-\frac{i}{n-1}%
x\right) ^{2}}\right) $, and we have left to prove the implication%
\[
\sum_{i=0}^{n-1}\frac{1}{1-\frac{i}{n-1}x}=\sum_{i=0}^{n-k-1}\frac{1}{1-x-%
\frac{i}{n-1}x}\Rightarrow \sum_{i=0}^{n-1}\frac{1}{\left( 1-\frac{i}{n-1%
}x\right) ^{2}}<\sum_{i=0}^{n-k-1}\frac{1}{\left( 1-x-\frac{i}{n-1}x\right)
^{2}}.
\]

Choosing $a_{i+1}=\frac{1}{1-\frac{i}{n-1}x}$, $i\in \left\{ 0,\ldots
,n-1\right\} $ and $b_{j+1}=\frac{1}{1-x-\frac{j}{n-1}x}$, $j\in \left\{
0,\ldots ,n-k-1\right\} $, we have $\max_{1\leq i\leq n}a_{i}=a_{n}=\frac{1}{%
1-x}=b_{1}=\min_{1\leq j\leq n-k}b_{j}$, and the above implication follows
from Lemma \ref{Reversed Cauchy inequality}, concluding the proof of the
claim.

We showed that $\varphi _{n,k}\left( x\right) =0$ implies $\varphi
_{n,k}^{^{\prime }}\left( x\right) <0$. Since $\varphi _{n,k}$ is
continuously differentiable, a moment's thought shows that this condition
implies that $\varphi _{n,k}$ can change signs at most once on the interval $%
\left[0,\frac{n-1}{2n-k-2}\right)$.

Since $\varphi _{n,k}\left( 0\right) =n-\left( n-k\right) =k>0$ and $%
\lim_{x\nearrow \frac{n-1}{2n-k-2}}\varphi _{n,k}\left( x\right) =-\infty $,
the function $\varphi _{n,k}$ changes sign on $\left[0,\frac{n-1}{2n-k-2}\right)$, and
let $x_{n,k}$ denote its unique root. We have left to show that $x_{n,k}$
belongs to the specified interval.

Using Lemma \ref{lemma for sum bounds} with $N=n-1$ and $%
f\left( t\right) =\frac{1}{1-tx}$, respectively with $N=n-k-1$ and $f\left(
t\right) =\frac{1}{1-x-t\frac{n-k-1}{n-1}x}$, we obtain:%
\begin{eqnarray*}
&&\varphi _{n,k}\left( x\right)
\leq \left( \left( n-1\right) \int_{0}^{1}\frac{1}{1-tx}dt+\frac{1+\frac{1%
}{1-x}}{2}+\frac{\frac{x}{\left( 1-x\right) ^{2}}-x}{4\left( n-1\right) }%
\right) \\
&&-\left( \left( n-k-1\right) \int_{0}^{1}\frac{1}{1-x-tx\frac{n-k-1}{%
n-1}}dt+\frac{\frac{1}{1-x}+\frac{1}{1-x-\frac{n-k-1}{n-1}x}}{2}\right) \\
&=&\left( -\frac{n-1}{x}\ln \left( 1-x\right) +\frac{1+\frac{1}{1-x}}{2}+%
\frac{\frac{x}{\left( 1-x\right) ^{2}}-x}{4\left( n-1\right) }\right)\\
&&-\left( -\frac{n-1}{x}\ln \frac{1-x-\frac{n-k-1}{n-1}x}{1-x}+\frac{\frac{1}{%
1-x}+\frac{1}{1-x-\frac{n-k-1}{n-1}x}}{2}\right) \\
&=&\frac{1}{2}\left( 1-\frac{1}{1-x-\frac{n-k-1}{n-1}x}\right) +\frac{%
x^{2}\left( 2-x\right) }{4\left( n-1\right) \left( 1-x\right) ^{2}}+\frac{n-1%
}{x}\ln \frac{1-x-\frac{n-k-1}{n-1}x}{\left( 1-x\right) ^{2}}.
\end{eqnarray*}

In particular, for $x=\frac{k}{n-1}$ we obtain $\varphi _{n,k}\left( \frac{k%
}{n-1}\right) \leq \frac{k\left( 2n-k-2\right) \left( k-2\left( n-1\right)
^{2}\right) }{4\left( n-1\right) ^{2}\left( n-k-1\right) ^{2}}<0$, which
shows that $x_{n,k}<\frac{k}{n-1}$.

In order to obtain the lower bound for $x_{n,k}$, first note that for $k=1$
the claim is trivial ($\varphi _{n,1}\left( 0\right) =1>0$, thus $x_{n,1}>0$%
), so we may assume $k\in \left\{ 2,\ldots ,n-1\right\} $. Using again Lemma %
\ref{lemma for sum bounds} with the same choices as above, we obtain%
\begin{eqnarray*}
\varphi _{n,k}\left( x\right)
&\geq&\left( \left( n-1\right) \int_{0}^{1}\frac{1}{1-tx}dt+\frac{1+\frac{1%
}{1-x}}{2}\right) -\left( \left( n-k-1\right) \int_{0}^{1}\frac{1}{1-x-tx%
\frac{n-k-1}{n-1}}dt\right.\\
&&\left.+\frac{\frac{1}{1-x}+\frac{1}{1-x-\frac{n-k-1}{n-1}x}}{2}%
+\frac{\frac{\frac{n-k-1}{n-1}x}{\left( 1-x-\frac{n-k-1}{n-1}x\right) ^{2}}-%
\frac{\frac{n-k-1}{n-1}x}{\left( 1-x\right) ^{2}}}{4\left( n-k-1\right) }%
\right) \\
&=&\left( -\frac{n-1}{x}\ln \left( 1-x\right) +\frac{1+\frac{1}{1-x}}{2}%
\right) -\left( -\frac{n-1}{x}\ln \frac{1-x-\frac{n-k-1}{n-1}x}{1-x}\right.\\
&&\left. +\frac{%
\frac{1}{1-x}+\frac{1}{1-x-\frac{n-k-1}{n-1}x}}{2}+\frac{\frac{\frac{n-k-1}{%
n-1}x}{\left( 1-x-\frac{n-k-1}{n-1}x\right) ^{2}}-\frac{\frac{n-k-1}{n-1}x}{%
\left( 1-x\right) ^{2}}}{4\left( n-k-1\right) }\right) \\
&=&\frac{1}{2}\left( 1-\frac{1}{1-x-\frac{n-k-1}{n-1}x}\right) -\left( \frac{%
1}{\left( 1-x-\frac{n-k-1}{n-1}x\right) ^{2}}-\frac{1}{\left( 1-x\right) ^{2}%
}\right) \frac{x}{4\left( n-1\right) }\\
&& +\frac{n-1}{x}\ln \frac{1-x-\frac{n-k-1%
}{n-1}x}{\left( 1-x\right) ^{2}}.
\end{eqnarray*}

To simplify the following computation, denote by $A=\left( \frac{n-k}{n-1}%
\right) ^{2}$, $B=\frac{k-1}{\left( n-1\right) ^{2}}$, and $C=\frac{A}{B}=%
\frac{\left( n-k\right) ^{2}}{k-1}$. For $x=\frac{k-1}{n-1}$, the above
becomes%
\begin{eqnarray*}
\varphi _{n,k}\left( x\right) \left( \frac{k-1}{n-1}\right) &\geq &\frac{1}{2%
}\left( 1-\frac{1}{A+B}\right) -\left( \frac{1}{\left( A+B\right) ^{2}}-%
\frac{1}{A^{2}}\right) \frac{B}{4}+\frac{1}{B}\ln \left( \frac{A+B}{A}%
\right) \\
&= & \frac12+ \frac{1}{B}\left(-\frac{1}{2\left( 1+C\right) }-\frac{1}{%
4\left( 1+C\right) ^{2}}+\frac{1}{4C^{2}}+\ln \left( 1+\frac{1}{C}\right)
\right) 
\end{eqnarray*}

Since
\begin{eqnarray*}
&&\frac{d}{dC}\left( -\frac{1}{2\left( 1+C\right) }-\frac{1}{4\left(
1+C\right) ^{2}}+\frac{1}{4C^{2}}+\ln \left( 1+\frac{1}{C}\right) \right)\\
&=&\left( \frac{1}{2\left( 1+C\right) ^{2}}-\frac{1}{C\left( 1+C\right) }%
\right) +\left( \frac{1}{2\left( 1+C\right) ^{3}}-\frac{1}{2C^{3}}\right)
<0,
\end{eqnarray*}%
and $\lim_{C\rightarrow \infty
}\left( -\frac{1}{2\left( 1+C\right) }-\frac{1}{4\left( 1+C\right) ^{2}}+%
\frac{1}{4C^{2}}+\ln \left( 1+\frac{1}{C}\right) \right) =0$, we conclude
that $$\left( -\frac{1}{2\left( 1+C\right) }-\frac{1}{4\left( 1+C\right) ^{2}}%
+\frac{1}{4C^{2}}+\ln \left( 1+\frac{1}{C}\right) \right) >0, \qquad \text{for all } C>0.$$

Using this and the previous inequality we obtain $x_{n,k}>\frac{k-1}{n-1}$, concluding the proof.
\end{proof}

\section{Main results}\label{Main results}

We can now prove the first main result, as follows.

\begin{theorem}\label{theorem on monotonicity of probabilities}
For arbitrarily fixed
integers $n\geq 2$ and $k\in \left\{ 0,1,\ldots ,n\right\} $, the
probability $$p_{n,k}\left( x\right) =p_{n,k}^{x,1-x,-\min \left\{
x,1-x\right\} /\left( n-1\right) }\left( x\right) $$ given by (\ref%
{probabilities p_n,k}) increases for $x\in \left[ 0,x_{n,k}^{\ast }\right] $
and decreases for $x\in \left[ x_{n,k}^{\ast },1\right] $, where
\begin{equation}
x_{n,k}^{\ast }=\left\{
\begin{tabular}{ll}
$x_{n,k},$ & if $k\leq \frac{n-1}{2}$ \\
$\frac{1}{2},$ & if $\frac{n-1}{2}<k<\frac{n+1}{2}$ \\
$1-x_{n,n-k},$ & if $k\geq \frac{n+1}{2}$%
\end{tabular}%
\right. ,  \label{definition of x^*}
\end{equation}%
and $x_{n,k}\in \left[ \frac{k-1}{n-1},\frac{k}{n-1}\right] $ are given by Lemma \ref{Localization lemma for x_n,k}.
\end{theorem}

\begin{proof}
First note that $p_{n,0}\left( x\right) =\prod_{i=0}^{n-1}\frac{\left(
1-x-i\min \left\{ x,1-x\right\} /\left( n-1\right) \right) }{1-i\min \left\{
x,1-x\right\} /\left( n-1\right) }$ is a decreasing function of $x\in \left[
0,1\right] $ (each factor decreases in $x$), and similarly $p_{n,n}\left(
x\right) =\prod_{i=0}^{n-1}\frac{\left( x-i\min \left\{ x,1-x\right\}
/\left( n-1\right) \right) }{1-i\min \left\{ x,1-x\right\} /\left(
n-1\right) }$ is an increasing function of $x\in \left[ 0,1\right] $ (each
factor increases in $x$). The claim of the theorem holds therefore true in
the cases $k=0$ and $k=n$ ($x_{n,0}^{\ast }=x_{n,0}=0$, respectively $%
x_{n,n}^{\ast }=1-x_{n,0}=1$), and we can assume that $k\in \left\{ 1,\ldots
,n-1\right\} $.

For $x\in \left( 0,1/2\right) $, we have%
\begin{eqnarray}
\frac{d}{dx}\ln p_{n,k}\left( x\right) &=&\frac{d}{dx}\left( \ln
C_{n}^{k}+\sum_{i=0}^{k-1}\ln \left( x-\frac{i}{n-1}x\right)\right. \label{first derivative} \\
&&\left.+\sum_{i=0}^{n-k-1}\ln \left( 1-x-\frac{i}{n-1}x \right) -\sum_{i=0}^{n-k-1}\ln \left( 1-\frac{i}{n-1}x\right) \right)\nonumber\\
&=&\frac{k}{x}+\sum_{i=0}^{n-k-1}\frac{-\left( 1+\frac{i}{n-1}\right) }{1-x-%
\frac{i}{n-1}x}-\sum_{i=0}^{n-1}\frac{-\frac{i}{n-1}}{1-\frac{i}{n-1}x}\nonumber \\
&=&\frac{k}{x}+\frac{1}{x}\sum_{i=0}^{n-k-1}\frac{1-\left( 1+\frac{i}{n-1}%
\right) x-1}{1-x-\frac{i}{n-1}x}-\frac{1}{x}\sum_{i=0}^{n-1}\frac{1-\frac{i}{%
n-1}x-1}{1-\frac{i}{n-1}x}  \nonumber \\
&=&\frac{k}{x}+\frac{n-k}{x}-\frac{1}{x}\sum_{i=0}^{n-k-1}\frac{1}{1-x-\frac{%
i}{n-1}x}-\frac{n}{x}+\frac{1}{x}\sum_{i=0}^{n-1}\frac{1}{1-\frac{i}{n-1}x}
\nonumber \\
&=&\frac{1}{x}\varphi _{n,k}\left( x\right) ,  \nonumber
\end{eqnarray}%
in the notation of Lemma \ref{Localization lemma for x_n,k}, and a similar
computation shows
\begin{equation}
\frac{d}{dx}\ln p_{n,k}\left( x\right) =-\frac{1}{1-x}\varphi _{n,n-k}\left(
1-x\right) ,\qquad x\in \left( 1/2,1\right) ,  \label{first derivative biss}
\end{equation}%
(alternatively, in order to derive the above one can use the relation $p_{n,k}\left(
x\right) =p_{n,n-k}\left( 1-x\right) $, valid for any $x\in \left[ 0,1\right]
$, $n\geq 2$, and $k\in \left\{ 0,1,\ldots ,n\right\} $).

It remains to show that the information about the sign of $\varphi
_{n,k}\left( x\right) $ given by Lemma \ref{Localization lemma for x_n,k}
translates into the monotonicity of $p_{n,k}\left( x\right) $ indicated in
the statement of the theorem.

Note that by Lemma \ref{Localization lemma for x_n,k} we have
\begin{equation}
x_{n,k}\in \left[ \frac{k-1}{n-1},\frac{k}{n-1}\right] ,\qquad \text{for all
}n\geq 2\text{ and }k\in \left\{ 1,\ldots ,n-1\right\} .
\label{localization of x}
\end{equation}

If $\frac{k}{n-1}\leq \frac{1}{2}$, Lemma \ref{Localization lemma for x_n,k}
and (\ref{first derivative}) show that $p_{n,k}$ increases on $\left[
0,x_{n,k}\right] $, and decreases on $\left[ x_{n,k},1/2\right] $ (note that
$x_{n,k}\leq \frac{1}{2}$ by (\ref{localization of x}) in this case). Since $%
x_{n,n-k}\geq \frac{n-k-1}{n-1}=1-\frac{k}{n-1}\geq \frac{1}{2}$, the
function $\varphi _{n,n-k}\left( x\right) $ is positive for $x\in \left[
0,1/2\right] \subset \left[ 0,x_{n,n-k}\right] $, and from (\ref{first
derivative biss}) it follows that $p_{n,k}$ decreases on $\left[ 1/2,1\right]
$. Since $p_{n,k}$ is a continuous function on $\left[ 0,1\right] $, it
follows that $p_{n,k}$ increases on $\left[ 0,x_{n,k}\right] $ and decreases
on $\left[ x_{n,k},1\right] $, and therefore the claim of the theorem holds
in this case with $x_{n,k}^{\ast }=x_{n,k}$.

If $\frac{k-1}{n-1}\geq \frac{1}{2}$, Lemma \ref{Localization lemma for
x_n,k} and (\ref{first derivative}) show that $p_{n,k}$ increases on $\left[
0,1/2\right] \subset \left[ 0,x_{n,k}\right] $ (note that $x_{n,k}\geq \frac{%
1}{2}$ by (\ref{localization of x}) in this case). Since $x_{n,n-k}\leq
\frac{n-k}{n-1}=1-\frac{k-1}{n-1}\leq \frac{1}{2}$, the function $\varphi
_{n,n-k}\left( x\right) $ is positive for $x\in \left[ 0,x_{n,n-k}\right]
\subset \left[ 0,1/2\right] $ and negative for $x\in \left[ x_{n,n-k},1/2%
\right] $, and from (\ref{first derivative biss}) it follows that $p_{n,k}$
increases on $\left[ \frac{1}{2},1-x_{n,n-k}\right] $, and decreases on $%
\left[ 1-x_{n,n-k},1\right] $. Since $f$ is continuous, it follows that $%
p_{n,k}$ increases on $\left[ 0,1-x_{n,n-k}\right] $ and decreases on $\left[
1-x_{n,n-k},1\right] $, and therefore the claim of the theorem holds in this
case with $x_{n,k}^{\ast }=1-x_{n,n-k}$.

We have left to consider the case $\frac{k-1}{n-1}<\frac{1}{2}<\frac{k}{n-1}$%
, or equivalent $\frac{n-1}{2}<k<\frac{n-1}{2}+1$. If $n$ is odd, the
previous double inequality is not satisfied for any integer $k$, so assume $%
n=2m$ is even. The previous double inequality gives $m-\frac{1}{2}<k<m+\frac{%
1}{2}$, which is satisfied only for $k=m$. We have
\[
\sum_{i=0}^{2m-1}\frac{1}{1-\frac{i}{n-1}\cdot \frac{1}{2}}%
>\sum_{j=0}^{m-1}\left( \frac{1}{1-\frac{2j}{n-1}\cdot \frac{1}{2}}+\frac{1}{1-\frac{2j}{n-1}\cdot \frac{1}{2}}\right) =2\sum_{j=0}^{m-1}\frac{1}{1-\frac{j}{n-1}},
\]%
and therefore
\[
\varphi _{2m,m}\left( \frac{1}{2}\right) =\sum_{i=0}^{2m-1}\frac{1}{1-\frac{i%
}{n-1}\cdot \frac{1}{2}}-\sum_{j=0}^{m-1}\frac{1}{\frac{1}{2}-\frac{j}{n-1}%
\cdot \frac{1}{2}}>0.
\]

Since $\varphi _{n,k}\left( x\right) =\varphi _{2m,m}\left( x\right)
=\varphi _{n,n-k}\left( x\right) $ for any $x\in \left[ 0,1\right] $, the
previous inequality shows that $x_{n,k}=x_{2m,m}=x_{n,n-k}>\frac{1}{2}$
(thus $\varphi _{n,k}=\varphi _{n,n-k}$ are positive on $\left[ 0,\frac{1}{2}%
\right] $). Using (\ref{first derivative}) and (\ref{first derivative biss})
we conclude that $p_{n,k}$ increases on $\left[ 0,\frac{1}{2}\right] $ and
decreases on $\left[ \frac{1}{2},1\right] $, thus the claim of the theorem
holds with $x_{n,k}^{\ast }=\frac{1}{2}$in this case, concluding the proof.
\end{proof}

\begin{remark}
\label{remark on localization of x^*}Since $p_{n,k}\left( x\right)
=p_{n,n-k}\left( 1-x\right) $ for $x\in \left[ 0,1\right] $, from the previous
theorem it follows that $x_{n,k}^{\ast }=1-x_{n,n-k}^{\ast }$, for all $%
n\geq 2$ and $k\in \left\{ 0,\ldots ,n\right\} $.

Also note that since $%
x_{n,k}\in \left[ \frac{k-1}{n-1},\frac{k}{n-1}\right] $, from (\ref%
{definition of x^*}) it follows that we also have $x_{n,k}^{\ast }\in \left[
\frac{k-1}{n-1},\frac{k}{n-1}\right] $ for all $n\geq 2$ and $k\in \left\{
0,1,\ldots ,n\right\} $.
\end{remark}

We are now ready to prove the main result. Recall that a random variables
variable $X$ is smaller than the random variable $Y$ in the usual stochastic
order (in symbols $X\leq _{\text{st}}Y$) if the corresponding distribution
functions $F_{X}$ and $F_{Y}$ satisfy $F_{X}\left( x\right) \geq F_{Y}\left(
x\right) $ for all $x\in \mathbb{R}$.

\begin{theorem}\label{Stochastic ordering of Polya rv}
For any $n\geq 2$, the random variables $X_{n}^{x,1-x,-\min \left\{ x,1-x\right\} /\left( n-1\right) }$ with Polya urn distribution given by (\ref%
{probabilities p_n,k}) satisfy the following stochastic ordering
\[
X_{n}^{x,1-x,-\min \left\{ x,1-x\right\} /\left( n-1\right) }\leq _{\text{st}%
}X_{n}^{y,1-y,-\min \left\{ x,1-y\right\} /\left( n-1\right) }, \qquad 0\leq x\leq y\leq 1.
\]%
\end{theorem}

\begin{proof}
Fix $n\geq 2$ and denote by $F_{x}\left( \cdot \right) $ the distribution
function of the random variable $X_{n}^{x,1-x,-\min \left\{ x,1-x\right\}
/\left( n-1\right) }$, $x\in \left[ 0,1\right] $. In order to prove the
claim, it suffices to show that for any $k\in \left\{ 0,1,\ldots ,n\right\} $%
, $F_{x}\left( k\right) $ is a decrea\-sing function of $x\in \left[ 0,1\right]
$, and we will prove this inductively on $k$.

Since $F_{x}\left( 0\right) =P\left( X_{n}^{x,1-x,-\min \left\{
x,1-x\right\} /\left( n-1\right) }=0\right) =p_{n,0}\left( x\right) $ is a
decreasing function of $x\in \left[ 0,1\right] $ (by Theorem \ref{theorem on
monotonicity of probabilities}), the claim holds true for $k=0$.

Assume now that the claim is true for $k-1$, i.e. $F_{x}\left( k-1\right) $
is decreasing in $x\in \left[ 0,1\right] $.

Theorem \ref{theorem on monotonicity of probabilities} shows that $%
p_{n,k}\left( x\right) $ is a decreasing function of $x\in \left[
x_{n,k}^{\ast },1\right] $, and therefore $F_{x}\left( k\right) =F_{x}\left(
k-1\right) +p_{n,k}\left( x\right) $ is decreasing for $x\in $ $\left[
x_{n,k}^{\ast },1\right] $.

Considering now $x\in \left[ 0,x_{n,k}^{\ast }\right] $, we observe that%
\begin{equation}
F_{x}\left( k\right) =\sum_{i=0}^{k}p_{n,i}\left( x\right)
=1-\sum_{i=k+1}^{n}p_{n,n-i}\left(
1-x\right) =1-\sum_{i=0}^{n-k-1}p_{n,i}\left( 1-x\right) .  \label{aux3}
\end{equation}

Using Remark \ref{remark on localization of x^*} we obtain $x_{n,k}^{\ast
}+x_{n,n-k-1}^{\ast }\leq \frac{k}{n-1}+\frac{n-k-1}{n-1}=1$, and it follows
that for $x\in \left[ 0,x_{n,k}^{\ast }\right] $ we have%
\begin{equation}
1-x\geq 1-x_{n,k}^{\ast }\geq x_{n,n-k-1}^{\ast }\geq x_{n,i}^{\ast },\qquad
i\in \left\{ 0,1,\ldots ,n-k-1\right\} ,  \label{aux4}
\end{equation}%
since by Remark \ref{remark on localization of x^*} we have $x_{n,i}^{\ast
}\leq \frac{i}{n-1}\leq x_{n,i+1}^{\ast }$ for all $i\in \left\{ 0,1,\ldots
n-1\right\} $.

Using (\ref{aux3}) and (\ref{aux4}), together with the monotonicity of $%
p_{n,i}$ given by Theorem \ref{theorem on monotonicity of probabilities}, it
follows that $F_{x}\left( k\right) $ is also decreasing for $x\in \left[
0,x_{n,k}^{\ast }\right] $, concluding the proof of the theorem.
\end{proof}


It is known (e.g. \cite{Shanthikumar}, p. 4) that the stochastic comparison $%
X\leq _{\text{st}}Y$ is equivalent to $Ef\left( X\right) \leq Ef\left(
Y\right) $ for all increasing functions $f$ for which the expectations exist.

Using this and the definition (\ref{operator R_n}) of the operator $R_{n}$, we can restate the above theorem as follows.

\begin{theorem}\label{R_n monotone operator}
The operator $R_{n}$ defined by (\ref{operator R_n}) is a monotone operator,
that is, if $f:\left[ 0,1\right] \rightarrow \mathbb{R}$ is monotone
increasing (decreasing), then $R_{n}(f,\cdot):\left[ 0,1\right] \rightarrow \mathbb{R%
}$ is also monotone increasing (decreasing).
\end{theorem}

\end{document}